\documentclass[reqno, 12pt]{amsart}
\usepackage{verbatim}
\usepackage{amssymb}
\usepackage{enumerate}
\usepackage[active]{srcltx}
\numberwithin{equation}{section}

\usepackage{t1enc}

\usepackage[utf8x]{inputenc}

\newtheorem{theorem}{Theorem}[section]

\newtheorem{lemma}[theorem]{Lemma}
\newtheorem{corollary}[theorem]{Corollary}

\newtheorem{problem}{Problem}[section]

\theoremstyle{definition}

\newcommand{\setm}{\setminus}
\newcommand{\empt}{\emptyset}
\newcommand{\subs}{\subset}

\def\<{\left\langle}
\def\>{\right\rangle}
\def\br#1;#2;{\bigl[ {#1} \bigr]^ {#2} }

\author[I. Juh\'asz]{Istv\'an Juh\'asz}
\address
      { Alfr\'ed R\'enyi Institute of Mathematics, Hungarian Academy of Sciences}
\email{juhasz@renyi.hu}

\author[J. van Mill]{Jan van Mill}
\address
      {University of Amsterdam}
\email{J.vanMill@uva.nl}

\title
   {On $\sigma$-countably tight spaces}

\date{\today}

\begin{document}

\begin{abstract}
Extending a result of R. de la Vega, we prove that an infinite homogeneous compactum has cardinality
$\mathfrak{c}$ if either it is the union of countably many {\it dense} or finitely
many arbitrary countably tight subspaces. The question if every
infinite homogeneous and $\sigma$-countably tight compactum has cardinality $\mathfrak{c}$ remains open.

We also show that if an arbitrary product is $\sigma$-countably tight
then all but finitely many of its factors must be countably tight.
\end{abstract}

\maketitle

\section{Introduction}

In \cite{dlV} R. de la Vega verified an old conjecture of Arhangel'skii by proving that
every infinite countably tight (in short: CT) homogeneous compactum has cardinality $\mathfrak{c}$.
The aim of this paper is to see whether in this result CT could be weakened to $\sigma$-CT, i.e.
if de la Vega's result remains valid when the homogeneous compactum is only assumed to be the
union of countably many CT subspaces. We conjecture that the answer to this question is affirmative
and provide results that, at least to us, convicingly point in this direction.

In fact, we shall prove below that an infinite homogeneous compactum has cardinality $\mathfrak{c}$
if either it is the union of {\em finitely many} CT subspaces or the union of countably many {\em dense} CT subspaces.

Just to see that the assumption of compactness is really essential in these types of results,
we mention here the following example. Consider the Cantor cube $\mathbb{C}_{\kappa} = \{0,1\}^\kappa$
and in it the subspaces $$\sigma_i = \big\{x \in \{0,1\}^\kappa : |\{\alpha < \kappa : x(\alpha) \ne i \}| < \omega \big\}$$
for $i \in \{0,1\}$. Then $\sigma_0 \cup \sigma_1$ is a $\sigma$-compact subgroup of $\mathbb{C}_\kappa$ that
is the union of two CT, even Fr\`echet, subspaces but, as is easily seen,
has tightness and cardinality $\kappa$, for any cardinal $\kappa$.

Of course, it is natural to raise the following question: Is there at all a homogeneous $\sigma$-CT compactum
that is not CT? Now, if our above conjecture is valid and $\mathfrak{c} < 2^{\omega_1}$ then, by the
\v Cech--Pospi\v sil theorem,  every homogeneous $\sigma$-CT compactum is even first countable, hence
any such example can only exist in a model in which $\mathfrak{c} = 2^{\omega_1}$.
On the other hand, we know that the answer to this question is trivially negative if homogeneity is dropped:
The compact ordinal space $\omega_1 + 1$ is the union of two CT subspaces but is not CT.

\smallskip

{\bf Acknowledgments}. This paper derives from the authors' collaboration at the Rényi Institute in Budapest in
the spring of 2016. The second-listed author is pleased to thank hereby the Hugarian Academy of Sciences for
its generous support in the framework of the distinguished visiting scientists program of the Academy
and the Rényi Institute for providing excellent conditions and generous hospitality.
The first author also thanks the support of the NKFIH grant no. 113047.

\section{Subseparable $G_\delta$-sets exist in $\sigma$-CT compacta}

In the proof of de la Vega's result on the size of homogeneous CT compacta a crucial role
was played by Arhangel'skii's observation that every CT compactum admits (non-empty)
subseparable $G_\delta$-sets. Naturally, we call a set subseparable if it is included in the closure of some
countable set. The main aim of this section is to show that this statement is also valid for
$\sigma$-CT compacta. To achieve this aim, we formulate and prove several auxiliary
lemmas.

For any space $X$ we shall denote by $\mathcal{G}(X)$ the family of all non-empty closed
$G_\delta$ subsets of $X$. Clearly, if $X$ is regular then for every $G_\delta$ subset $H$ of $X$
we have $H = \bigcup \{G \in \mathcal{G}(X) : G \subset H \}$.

\begin{lemma}\label{lm:CCdense}
Let $D$ be a countably compact dense subset of the normal space $X$. Then for every countable cover
$\{A_n : n < \omega\}$ of $X$ there are $H \in \mathcal{G}(X)$ and $n < \omega$ such that
$H \cap D \cap A_n$ is $G_\delta$-dense in $H$.
\end{lemma}

\begin{proof}
Let us note first that $D$ is actually $G_\delta$-dense in $X$. Indeed, the normality of $X$
implies that any $H \in \mathcal{G}(X)$ is of the form $H = \bigcap \{U_n : n < \omega\}$
where $U_n$ is open and $\overline{U_{n+1}} \subs U_n$ for all $n < \omega$. Thus,
if we pick $x_n \in D \cap U_n$
then any accumulation point of the sequence $\{x_n : n < \omega \}$ is in $H$,
hence $D \cap H \ne \empt$ because $D$ is countably compact.

Assume next that the conclusion of our lemma is false.
Then we can find a decreasing sequence $\{H_n : n < \omega \} \subset \mathcal{G}(X)$ such that
$H_n \cap D \cap A_n = \empt$ for all $n < \omega$. But then $\bigcap \{H_n \cap D : n < \omega\} \ne \empt$ as $D$
is countably compact, contradicting that  $\{A_n : n < \omega\}$ covers $X$.
\end{proof}

\begin{lemma}\label{lm:pureG_d}
Let $X$ be a countably compact regular space and $\{A_n : n < \omega\}$ be any countable cover of $X$.
Then there is $H \in \mathcal{G}(X)$ such that, for every $n < \omega$, if $A_n \cap H \ne \empt$
then $A_n \cap H$ is $G_\delta$-dense in $H$.
\end{lemma}

\begin{proof}
Starting with $H_0 = X$, we may define by a straight forward recursion sets $H_n \in \mathcal{G}(X)$
for all $n < \omega$ such that if $A_n \cap H_n$ is not $G_\delta$-dense in $H_n$ then
$H_{n+1} \subset H_n$ and $A_n \cap H_{n+1} = \empt$.
Clearly, then $H = \bigcap \{H_n : n < \omega \} \in \mathcal{G}(X)$ is as required.
\end{proof}

\begin{lemma}\label{lm:w-bd}
If $X$ is a $\sigma$-CT compactum then any pairwise disjoint collection of
dense $\omega$-bounded subspaces of $X$ is countable.
\end{lemma}

\begin{proof}
We have a countable cover $\{A_n : n < \omega\}$ of $X$ where each $A_n$ is CT.
Assume now that $\{D_\alpha : \alpha < \omega_1 \}$ are dense $\omega$-bounded subspaces of $X$.
Applying lemma \ref{lm:CCdense}, we may then define by transfinite recursion
on $\alpha < \omega_1$ sets $H_\alpha \in \mathcal{G}(X)$
and natural numbers $n_\alpha < \omega$ such that
\begin{enumerate}[(1)]
\item if $\beta < \alpha$ then $H_\alpha \subs H_\beta$,

\item  $H_\alpha \cap D_\alpha \cap A_{n_\alpha}$ is $G_\delta$-dense in $H_\alpha$.
\end{enumerate}

Now, pick distinct $\beta < \alpha < \omega_1$ such that $n_\beta = n_\alpha = n$. Then,
since $A_n$ is CT and  $H_\alpha \cap D_\alpha$ is $\omega$-bounded, we have $\empt \ne A_n \cap H_\alpha \subs D_\alpha$,
and similarly $A_n \cap H_\beta \subs D_\beta$.  But $\empt \ne A_n \cap H_\alpha \subs A_n \cap H_\beta$,
and this clearly implies $D_\alpha \cap D_\beta \ne \empt$.
\end{proof}

It is easy to see that in the Cantor cube $\mathbb{C}_{\omega_1} = \{0,1 \}^{\omega_1}$
there are uncountably many (in fact, $2^{\omega_1}$ many) pairwise disjoint
dense $\omega$-bounded subspaces of the form $f + \Sigma$, where $\Sigma$ is
the subgroup of  $\mathbb{C}_{\omega_1}$ consisting of all its members having
countable support. Also, it is obvious that if $\pi : X \to Y$ is an irreducible
continuous map between compacta then for every dense $\omega$-bounded subspace $D$ of $Y$
its inverse image $\pi^{-1}[D]$ is a dense $\omega$-bounded subspace of $X$.
Thus we immediately obtain the following corollary of lemma \ref{lm:w-bd} which,
of course, is well-known for CT compacta.

\begin{lemma}\label{lm:C}
If $X$ is a $\sigma$-CT compactum then no closed subspace of $X$ can be mapped
onto $\mathbb{C}_{\omega_1}$. In particular, then every non-empty closed subspace $Y$
has a point $y \in Y$ with $\pi\chi(y,Y) \le \omega$.
\end{lemma}

We are now ready to present the main result of this section.

\begin{theorem}\label{tm:subsep}
Every $\sigma$-CT compactum $X$ has a non-empty subseparable $G_\delta$ subset.
\end{theorem}

\begin{proof}
If $X$ has a $G_\delta$ point, i.e. a point of first countability then we are done.
So, we may assume that $X$ is nowhere first countable that clearly implies that
$\chi(x,H) > \omega$ whenever $x \in H \in \mathcal{G}(X)$. Then $|H| \ge 2^{\omega_1}$
by the \v Cech-Pospi\v sil theorem.

We have $X = \bigcup \{A_n : n < \omega \}$ where every $A_n$ is CT.
By lemma \ref{lm:pureG_d} we may also assume without any loss of generality that
every $A_n$ is $G_\delta$-dense in $X$.

Now, assume that no member of $\mathcal{G}(X)$ is subseparable and by transfinite recursion
on $\alpha < \omega_1$ define $S_\alpha, T_\alpha \in \mathcal{G}(X)$, points $x_\alpha \in S_\alpha$,
and countable sets $B^n_\alpha$ for $n < \omega$ such that the following
inductive hypotheses hold:

\begin{enumerate}[(1)]
\item $S_\alpha \cap T_\alpha = \empt$,

\item $T_\alpha \subs \bigcap \{T_\beta : \beta < \alpha\}$,

\item $B^n_\beta \subs S_\alpha$ for any $\beta < \alpha$ and $n < \omega$,

\item  $B^\alpha_n \subs A_n \cap S_\alpha \cap \bigcap \{T_\beta : \beta < \alpha\}$ and
$x_\alpha \in \overline{B^n_\alpha}$ for all $n < \omega$.
\end{enumerate}

$S_0,\,T_0$ are any two disjoint members of  $\mathcal{G}(X)$ and $x_0 \in S_0$
is chosen to satisfy $\pi\chi(x_0, S_0) = \omega$; this is possible by lemma \ref{lm:C}.
This implies the existence of a countable set $B^n_0 \subs A_n \cap S_0 $
with $x_0 \in \overline{B^n_0}$ for each $n < \omega$ because $A_n$ is dense in $S_0$.
If $0 < \alpha < \omega_1$ and the construction has been completed for all $\beta < \alpha$,
put $T = \bigcap \{T_\beta : \beta < \alpha\}$ and $B = \bigcup \{B^\beta_n : \beta< \alpha, \,n < \omega\}$.
Then $B$ is countable, hence $T \setm \overline{B}$ is a non-empty $G_\delta$ by our indirect
assumption. Consequently, there are disjoint  $H,K \in \mathcal{G}(X)$ such that $H \subs T$
and $\overline{B} \subs K$. Next we may choose disjoint sets $H_0,\,H_1 \in \mathcal{G}(H) \subs \mathcal{G}(X)$
and the point $x_\alpha \in H_0$  with $\pi\chi(x_\alpha, H_0) = \omega$.
Then again we have a countable set $B^n_\alpha \subs A_n \cap H_0$
with $x_\alpha \in \overline{B^\alpha_n}$ for each $n < \omega$. Then, putting $S_\alpha = K \cup H_0$
and $T_\alpha  = H_1$, it is easy to see that the inductive hypotheses remain valid,
completing the recursive construction.

Let $x$ be a complete accumulation point of the set $\{x_\alpha : \alpha < \omega_1\}$. Then
there is $n < \omega$ for which $x \in A_n$, moreover (4) implies both $x \in \bigcap_{\alpha<\omega_1}T_\alpha$
and $x \in \overline{\bigcup_{\alpha<\omega_1}B^n_\alpha}$. Consequently, as $A_n$ is CT, there is some $\alpha<\omega_1$
such that $x \in \overline{\bigcup_{\beta<\alpha}B^n_\beta}$, hence $x \in S_\alpha$ by (3). But this would
imply $x \in S_\alpha \cap T_\alpha$, contradicting (1).

\end{proof}

\section{A "two cover" theorem}

A subseparable subspace of a regular space has weight $\le \mathfrak{c}$, so in view of the previous
section any $\sigma$-CT compactum has many $G_\delta$ sets of weight $\le \mathfrak{c}$.
The result we prove in this section, however, needs more: having a cover of the space by
$G_\delta$ sets of weight $\le \mathfrak{c}$. Of course, if the space in question is also
homogeneous then the existence of a non-empty $G_\delta$ set of weight $\le \mathfrak{c}$
implies the existence of such a cover. Also, being $\sigma$-CT just means that our space
has a countable cover by CT sets. Thus we have the two covers referred to in the title of
this section.

\begin{theorem}\label{tm:Lind}
Let $X$ be a Lindel\"of regular space with two covers $\mathcal{Y}$ and $\mathcal{H}$ such that
\begin{enumerate}[(1)]
\item $|\mathcal{Y}| \le \mathfrak{c}$, moreover every $Y \in \mathcal{Y}$ is CT and satisfies 
$$X = \bigcup \{\overline{A} : A \in [Y]^{\le\mathfrak{c}}\}\,;$$

\item $\mathcal{H} \subs \mathcal{G}(X)$ and $w(H) \le \mathfrak{c}$ for every $H \in \mathcal{H}$;

\item for every set $D \in [X]^{\le \mathfrak{c}}$ we have $w(\overline{D}) \le \mathfrak{c}$.
\end{enumerate}
Then $w(X) \le \mathfrak{c}$.

\end{theorem}

The proof of this theorem will be based on the following two rather general lemmas.
The first one deals with a cover $\mathcal{Y}$ as in (1) and the second with a cover
like $\mathcal{H}$ in (2).

\begin{lemma}\label{lm:Y}
Let $X$ be any space with a cover $\mathcal{Y}$ exactly as in (1) above, moreover
assume that the closure $\overline{D}$ of every set $D \in [X]^{\le \mathfrak{c}}$
is Lindel\"of and has pseudocharacter $\psi(\overline{D},X) \le \mathfrak{c}$.
Then $d(X) \le \mathfrak{c}$.
\end{lemma}

\begin{proof}
We shall say that a set $S \subs X$ is $\mathcal{Y}$-saturated if $Y \cap S$ is dense in $S$
for every $Y \in \mathcal{Y}$. Obviously, any union of $\mathcal{Y}$-saturated  sets is $\mathcal{Y}$-saturated.
It is clear from (1) that for every point $x \in X$ we may fix
a  $\mathcal{Y}$-saturated set $S(x) \in [X]^\mathfrak{c}$ with $x \in S(x)$.

We may also fix for every set $D \in [X]^{\le \mathfrak{c}}$ a collection $\mathcal{U}(D)$
of open sets with $|\mathcal{U}(D)| \le \mathfrak{c}$ such that $\cap\, \mathcal{U}(D) = \overline{D}$.

Next, by transfinite recursion on $\alpha < \omega_1$ we define $\mathcal{Y}$-saturated sets $D_\alpha \in [X]^{\mathfrak{c}}$
as follows. We start by choosing $D_0$ as an arbitrary $\mathcal{Y}$-saturated set of size $\mathfrak{c}$. (If $|X| \le \mathfrak{c}$
then we are done.) Also, if $\alpha$ is limit then we simply put $D_\alpha = \bigcup_{\beta<\alpha}D_\beta$.

If $D_\alpha$ has been defined then in the successor case $\alpha+1$ we first consider the collection
$\mathcal{V}_\alpha = \bigcup \{\mathcal{U}(D_\beta) : \beta \le \alpha\}$ and then put
$$\mathcal{W}_\alpha = \{ \cup \mathcal{V} : \mathcal{V} \in \big[\mathcal{V}_\alpha \big]^{\le \omega}
\mbox{ and }\, X \setm  \cup \mathcal{V} \ne \empt \}.$$
Clearly, we have $|\mathcal{W}_\alpha| \le \mathfrak{c}$. For each $W \in \mathcal{W}_\alpha$ we may then
fix a point $x_W \in X \setm W$ and then put
$$D_{\alpha+1} = D_\alpha \cup \bigcup \{S(x_W) : W \in \mathcal{W}_\alpha \}.$$
Finally, we put $D = \bigcup_{\alpha< \omega_1} D_\alpha$, then we clearly have
$|D| = \mathfrak{c}$. We shall now show that $D$ is dense in $X$.

\smallskip

{\bf Claim 1.} $\overline{D} = \bigcup_{\alpha < \omega_1} \overline{D_\alpha}$.

\smallskip

Indeed, for any $x \in \overline{D}$ there is $Y \in \mathcal{Y}$ with $x \in Y$ and,
since $D$ is $\mathcal{Y}$-saturated, this implies $x \in  \overline{Y \cap D}$.
This, in turn, implies that there is a countable subset $A \subs Y \cap D$ with
$x \in \overline{A}$ because $Y$ is CT. But then there is some $\alpha < \omega_1$
for which $A \subs D_\alpha$, hence $x \in \overline{D_\alpha}$.

\smallskip

The following claim then finishes the proof.

\smallskip

{\bf Claim 2.} $X = \overline{D}$.

\smallskip

Assume that $x \in X \setm \overline{D}$. Then for each $\alpha < \omega_1$ there is
$U_\alpha \in \mathcal{U}(D_\alpha)$ with $x \notin U_\alpha$. But then $\{U_\alpha : \alpha < \omega_1 \}$
is an open cover of the Lindel\"of subspace $\overline{D}$, hence there is a countable ordinal $\gamma < \omega_1$
such that $W = \bigcup_{\alpha < \gamma} \supset \overline{D}$ as well. But then we also have $W \in \mathcal{W}_\gamma$,
hence $x_W \in D_{\gamma+1} \subs D$, contradicting that $x_W \notin W \supset D$.

\end{proof}

\begin{lemma}\label{lm:psi}
Let $X$ be a regular space and assume that $Z \subs X$ is a Lindel\"of subspace of weight $w(Z) \le \mathfrak{c}$, moreover
$Z$ admits a cover $\mathcal{H} \subs \mathcal{G}(X)$ with $w(H) \le \mathfrak{c}$ for every $H \in \mathcal{H}$.
Then $\psi(Z,X) \le \mathfrak{c}$.
\end{lemma}
\begin{proof}
We first show that $w(Z) \le \mathfrak{c}$ implies $|\mathcal{G}(Z)| \le \mathfrak{c}$.
So we fix an open base $\mathcal{B}$ of $Z$ with $|\mathcal{B}| \le \mathfrak{c}$.
Every set $S \in \mathcal{G}(Z)$ is then the intersection of a countable family $\mathcal{U}$
of sets open in $Z$. Now $S$, being closed in $Z$, is also Lindel\"of, hence for every $U \in \mathcal{U}$
there is a countable subfamily $\mathcal{B}_U$ of $\mathcal{B}$ such that
$S \subs B_U = \cup \mathcal{B}_U \subs U$, consequently we have $S = \cap \{B_U : U \in \mathcal{U} \}$ as well.
Thus we conclude that $$|\mathcal{G}(Z)| \le \big| \big[ [\mathcal{B}]^\omega \big]^\omega\,\big| = \mathfrak{c}\,.$$
Of course, this means that we may assume without any loss of generality that $|\mathcal{H}| \le \mathfrak{c}$ as well.

Clearly,  the regularity of $X$ and $w(H) \le \mathfrak{c}$ imply $\psi(H \cap Z,\,H) \le \mathfrak{c}$ for each
$H \in \mathcal{H}$, but then $\psi(H \cap Z,\,X) \le \mathfrak{c}$ as well, since $H$ is a $G_\delta$.
So we may fix, for every $H \in \mathcal{H}$, a family $\mathcal{V}_H$ of open sets in $X$ with
$|\mathcal{V}_H| \le \mathfrak{c}$ such that $\cap \mathcal{V}_H = H \cap Z$. Then $|\mathcal{H}| \le \mathfrak{c}$
implies that $\mathcal{V} = \bigcup \{\mathcal{V}_H : H \in \mathcal{H}\}$ has cardinality $\le \mathfrak{c}$ as well.
Finally, we put $$\mathcal{W} = \{\cup \mathcal{V}' : \mathcal{V}' \in [\mathcal{V}]^{\le \omega} \mbox{ and }
Z \subs \cup \mathcal{V}'\}\,.$$
Clearly, we have $|\mathcal{W}| \le \mathfrak{c}$ as well.

We claim that $Z = \cap \mathcal{W}$, hence $\psi(Z,X) \le \mathfrak{c}$. To see this, pick any point $x \in X \setm Z$.
Then for each $H \in \mathcal{H}$ there is a member $V_H \in \mathcal{V}_H$ such that $x \notin V_H$.
The Lindel\"of property of $Z$ implies that $\mathcal{H}$ has a countable subfamily $\mathcal{H}'$
such that $\mathcal{V}' = \{V_H : H \in \mathcal{H}'\}$ covers $Z$. But then $W = \cup \mathcal{V}' \in \mathcal{W}$
and clearly $x \notin W$.
\end{proof}

We are now ready to give the proof of theorem \ref{tm:Lind}. First observe that
condition (3) of the theorem together
with lemma \ref{lm:psi} implies $\psi(\overline{D},X) \le \mathfrak{c}$ whenever $D \in [X]^{\le \mathfrak{c}}$.
This, however, means that $X$ satisfies (with $\mathcal{Y}$) all the conditions of lemma \ref{lm:Y},
hence we have a dense set  $D$ in $X$ of size  $\le \mathfrak{c}$. But this, in turn, implies $w(X) = w(\overline{D}) \le \mathfrak{c}$,
completing the proof of theorem \ref{tm:Lind}.

For further use in the next section, we present one more lemma.

\begin{lemma}\label{lm:nw}
Assume that $X$ is a regular space and $\mathcal{Y}$ is a cover of $X$ as in (1) of theorem \ref{tm:Lind}.
Then for every $D \in [X]^{\le \mathfrak{c}}$ we have $nw(\overline{D}) \le \mathfrak{c}$.
\end{lemma}

\begin{proof}
Since every subset of $X$ of size $\le \mathfrak{c}$ is included in a $\mathcal{Y}$-saturated
subset of of size $\le \mathfrak{c}$, we may assume without loss of generality that $D$ is
$\mathcal{Y}$-saturated. Now, we claim that the family $\mathcal{N} = \{\overline{A} : A \in [D]^\omega \}$ is a network
for $\overline{D}$.

Indeed, assume that $x \in \overline{D}$ and $U$ is any open set containing $x$.
Choose an open $V$ such that $x \in V \subs \overline{V} \subs U$. There is some $Y \in \mathcal{Y}$
with $x \in Y$ and the $\mathcal{Y}$-saturatedness of $D$ then implies $x \in \overline{V \cap D \cap Y}$.
But then, as $Y$ is CT, there is a countable set $A \subs V \cap D \cap Y$ such that $x \in \overline{A}$.
This clearly implies $x \in \overline{A} \subs \overline{V} \subs U$, which shows that $\mathcal{N}$
is a network for $\overline{D}$.
\end{proof}

Since for every compactum $X$ we have $nw(X) = w(X)$, this allows us to obtain the following
simplified form of theorem \ref{tm:Lind} for compact spaces.

\begin{corollary}\label{co:w}
Let $X$ be a compactum with covers $\mathcal{Y}$ and $\mathcal{H}$ such that
\begin{enumerate}[(1)]
\item $|\mathcal{Y}| \le \mathfrak{c}$ and every $Y \in \mathcal{Y}$ is CT and dense in $X$;

\item $\mathcal{H} \subs \mathcal{G}(X)$ and $w(H) \le \mathfrak{c}$ for every $H \in \mathcal{H}$.
\end{enumerate}

Then $w(X) \le \mathfrak{c}$.
\end{corollary}

\begin{proof}
Let us start by noting that if $x \in H \in \mathcal{H}$ then, by compactness, we have
$\chi(x,X) = \psi(x,X) \le \psi(x,H) \cdot \omega \le \mathfrak{c}$,
hence $\chi(X) \le \mathfrak{c}$. But this clearly implies
$X = \bigcup \{\overline{A} : A \in [Y]^{\le\mathfrak{c}}\}$ for every dense $Y \subs X$.
Consequently, $\mathcal{Y}$ satisfies all requirements of (1) from theorem \ref{tm:Lind},
hence by lemma \ref{lm:nw} we have  $nw(\overline{D}) = w(\overline{D}) \le \mathfrak{c}$
for every $D \in [X]^{\le \mathfrak{c}}$. But this is just condition (3) of theorem \ref{tm:Lind}
that implies  $w(X) \le \mathfrak{c}$.

\end{proof}

We close this section by acknowledging that the method we used to prove the results
of this section was motivated by the proof of
theorem 6.4 in \cite{AvM}. There, in turn, the authors give credit to the approach that
R. Buzyakova used in \cite{Bu}.

\section{Adding homogeneity}

Although we could not prove that all infinite homogeneous $\sigma$-CT compacta are of size $\mathfrak{c}$,
the results of this section provide significant steps in that direction.

\begin{theorem}\label{tm:DCT}
Assume that the compactum $X$ is the union of countably many {\em dense} CT subspaces,
moreover $X^\omega$ is homogeneous. Then $|X| \le \mathfrak{c}$.
\end{theorem}
\begin{proof}
Let $\mathcal{Y}$ be a countable family of dense CT subspaces of $X$ that covers $X$. Then,
by lemma \ref{lm:C}, there is a point $x \in X$ with $\pi\chi(x,X) \le \omega$, consequently
$X^\omega$ also has a point of countable $\pi$-character, namely the point all of whose
co-ordinates are equal to $x$. But $X^\omega$ is homogeneous, hence we actually have
$\pi\chi(X^\omega) = \omega$.

Next, applying theorem \ref{tm:subsep} we obtain the existence of some $G \in \mathcal{G}(X)$
that is subseparable and hence has weight $w(G) \le \mathfrak{c}$. But then we also have
$G^\omega \in \mathcal{G}(X^\omega)$, moreover $w(G^\omega) \le \mathfrak{c}$ as well.
Now, the homogeneity of $X^\omega$ then implies that actually $X^\omega$ can be covered by
closed $G_\delta$-sets of weight $\le \mathfrak{c}$. Consequently, this is also true for
$X$, i.e. there is a cover $\mathcal{H} \subs \mathcal{G}(X)$ of $X$ such that $w(H) \le \mathfrak{c}$
for all $H \in \mathcal{H}$.

Thus the two covers $\mathcal{Y}$ and $\mathcal{H}$ of $X$ satisfy both conditions (1) and (2)
of corollary \ref{co:w}, hence we can apply it to conclude that $w(X) \le \mathfrak{c}$ that, in turn,
implies $w(X^\omega) \le \mathfrak{c}$ as well. But it was shown in \cite{vM}
that any homogeneous compactum $Z$ satisfies the inequality $|Z| \le w(Z)^{\pi\chi(Z)}$,
consequently, we conclude that $|X| \le |X^\omega| \le \mathfrak{c}$. (Of course, we have $|X^\omega| = \mathfrak{c}$,
unless $X$ is a singleton.)
\end{proof}

In our next result we can get rid of the annoying condition of density for the members of $\mathcal{Y}$,
however we have to pay a price: the cover of $X$ by CT subspaces needs to be {\em finite}.
Also, we need to assume that $X$ itself, and not just $X^\omega$, is homogeneous.

\begin{theorem}\label{tm:finCT}
If $X$ is an infinite homogeneous compactum that is the union of finitely many CT subspaces then
$|X| = \mathfrak{c}$.
\end{theorem}

\begin{proof}
We start with the trivial remark that $|X| \ge \mathfrak{c}$ for any
infinite homogeneous compactum.

Following the arguments in the previous proof of theorem \ref{tm:DCT},
but using now the homogeneity of $X$,
we may conclude that $\pi\chi(X) = \omega$, moreover
there is a cover $\mathcal{H} \subs \mathcal{G}(X)$ of $X$ such that $w(H) \le \mathfrak{c}$
for all $H \in \mathcal{H}$.

Of course, we also have the finite cover $\mathcal{Y}$ of $X$ by CT subspaces that may not be dense.
But it is
straight forward then to find a non-empty open subset $U$ of $X$ such that for every
$Y \in \mathcal{Y}$ we have $Y \cap U$ is dense in $U$ whenever $Y \cap U \ne \empt$.
We put then $\mathcal{Z} = \{Y \in \mathcal{Y} : Y \cap U \ne \empt \}$.
Consider any non-empty open subset $V$ of $U$ such that $\overline{V} \subs U$.
But then corollary \ref{co:w} can be applied to $\overline{V}$ and
the covers $\mathcal{Z}$ and $\mathcal{H}$ restricted to $\overline{V}$ to conclude
that $w(V) \le w(\overline{V}) \le \mathfrak{c}$.

Now, using the homogeneity and the compactness of $X$, we can cover $X$ by finitely
many open sets each homeomorphic to $V$, that clearly implies $w(X) = w(V) \le \mathfrak{c}$.
Thus we are done because we have $|X| \le w(X)^{\pi\chi(X)} = \mathfrak{c}$, again by \cite{vM}.
\end{proof}

\smallskip

As is well-known, $\mathfrak{c} = 2^\omega < 2^{\omega_1}$ implies that any homogeneous compactum
of size $\mathfrak{c}$ is first countable. Consequently, under this assumption, the compacta
figuring in theorems \ref{tm:DCT} and \ref{tm:finCT} all turn out to be first countable, hence CT.
This fact makes the following natural problem even more interesting.

\begin{problem}
Is it consistent to have a homogeneous compactum that is $\sigma$-CT but not CT?
\end{problem}

\section{$\sigma$-CT products}

Since in theorem \ref{tm:DCT} one requires the homogeneity of $X^\omega$ instead of $X$,
it is natural to raise the question: What if, similarly, we require the $\sigma$-CT property from
$X^\omega$ rather than from $X$?

Now, the main aim of this section is to prove that if a product of ($T_1$) spaces is $\sigma$-CT
then all but finitely many of its factors are actually CT. This clearly implies that if $X^\omega$
is $\sigma$-CT then $X$ is actually CT.

We shall in fact prove a stronger
result for which we need the following lemma. We recall that a non-empty subset $S \subs X$ of a space
$X$ is called a {\em weak P set} if for every countable subset $T$ of its complement $X \setm S$
we have $\overline{T} \cap S = \empt$.

\begin{lemma}\label{lm:wP}
Consider the product $X \times Y$ where $X$ has a nowhere dense weak P subset $S$. Then no
{\em dense} CT subset $A$ of  $X \times Y$ intersects the subproduct $S \times Y$.
\end{lemma}

\begin{proof}
Indeed, then $S \times Y$ is clearly a nowhere dense weak P subset of $X \times Y$. Consequently,
$B = A \cap (X \times Y \setm S \times Y)$ is also dense in $X \times Y$ and no point of $S \times Y$
is in the closure of a countable subset of $B$. But this clearly implies that no point of $S \times Y$
can be in $A$ because it is CT.
\end{proof}

\begin{theorem}\label{tm:wP}
Assume that $\{X_i : i  < \omega \}$ is a sequence of spaces such that each $X_i$ has
a nowhere dense weak P subset $S_i$. Then their product $X = \prod \{X_i : i < \omega \}$
is not $\sigma$-CT.
\end{theorem}

\begin{proof}
Let $\{A_n : n < \omega \}$ be any countable collection of CT subspaces of $X$. We shall show
that $\{A_n : n < \omega \}$ does not cover $X$. To do that, we are going to define a
strictly increasing sequence $\{k_n : n < \omega\}$ of natural numbers and, for each $n < \omega$,
points $x_i \in X_i$ for $i < k_n$ such that if we put $Y_i = \{x_i\}$ for $i < k_n$ and $Y_i = X_i$
for $i \ge k_n$ then we have $$Z_n = \prod\{Y_i : i < \omega\} \subs X \setm \bigcup_{m<n} A_m = \empt\,.$$

To start with, we simply put $k_0 = 0$. Next, if $k_n$ and $x_i \in X_i$ for $i < k_n$ have been
chosen for some $n < \omega$, then we have to define $k_{n+1} > k_n$ and the points $x_i \in X_i$
for $k_n \le i < k_{n+1}$. To do that, we distinguish two cases.

\smallskip

{\bf Case 1.} $A_n \cap Z_n$ is not dense in $Z_n$. Then we can find $k_{n+1} > k_n$ and a non-empty
open set $U_i \subs X_i$ for all $k_n \le i < k_{n+1}$ such that $A_n$ is disjoint from the subproduct
of $Z_n$ that is obtained by shrinking $X_i$ to $U_i$ for $k_n \le i < k_{n+1}$ and leaving all other factors
unchanged. Thus if we pick $x_i \in U_i$ for $k_n \le i < k_{n+1}$ then the inductive hypothesis will
clearly remain valid for $n+1$.

\smallskip

{\bf Case 2.} $A_n \cap Z_n$ is dense in $Z_n$. In this case we put $k_{n+1} = k_n + 1$ and pick $x_{k_n} \in S_{k_n}$.
Then lemma \ref{lm:wP} implies that $Z_{n+1} \cap A_n = \empt$ and, as $Z_{n+1} \subs Z_n$, we are again done.

\smallskip

Having completed the induction it is obvious that the point of the product $X$ whose $i$th co-ordinate
is $x_i$ for all $i < \omega$ does not belong to $\bigcup_{n<\omega} A_n$, hence, indeed,
$\{A_n : n < \omega \}$ does not cover $X$.
\end{proof}

\begin{corollary}\label{co:prod}
If the product $X = \prod \{X_i : i \in I \}$ is $\sigma$-CT then only finitely many of its factors $X_i$
can have a nowhere dense weak P subset. In particular, all but finitely many of its factors $X_i$ are CT.
\end{corollary}

\begin{proof}
Indeed, if infinitely many factors $X_i$ would have a nowhere dense weak P subset then $X$ would
contain a subspace homeomorphic to a countably infinite product as in theorem \ref{tm:wP}, which is
clearly impossible.

The second part follows because it is clear that any space $Y$ that is not CT has a subspace
$Z$ containing a point $z \in Z$ such that the
singleton $\{z\}$ is a nowhere dense weak P subset of $Z$.
\end{proof}

\end{document}